\documentclass[a4paper,10pt]{article}

\usepackage[english]{babel}
\usepackage{amsmath,amsfonts,amssymb,amsthm}
\usepackage{mathrsfs}
\usepackage{indentfirst}

\newtheorem{theorem}{Theorem}[section]

\newtheorem{lemma}{Lemma}[section]

\newtheorem{corollary}{Corollary}[section]

\theoremstyle{definition}
\newtheorem{remark}{Remark}[section]

\numberwithin{equation}{section}

\newcommand\blfootnote[1]{\begingroup\renewcommand\thefootnote{}\footnote{#1}\addtocounter{footnote}{-1}\endgroup}

\begin{document}

\title{
{\bf\Large
Pairs of positive periodic solutions of nonlinear ODEs with indefinite weight: a topological degree approach
for the super-sublinear case }\footnote{Work performed under the auspices of the
Grup\-po Na\-zio\-na\-le per l'Anali\-si Ma\-te\-ma\-ti\-ca, la Pro\-ba\-bi\-li\-t\`{a} e le lo\-ro
Appli\-ca\-zio\-ni (GNAMPA) of the Isti\-tu\-to Na\-zio\-na\-le di Al\-ta Ma\-te\-ma\-ti\-ca (INdAM).}}

\author{
{\bf\large Alberto Boscaggin}
\vspace{1mm}\\
{\it\small Department of Mathematics, University of Torino}\\
{\it\small via Carlo Alberto 10}, {\it\small 10123 Torino, Italy}\\
{\it\small e-mail:  alberto.boscaggin@unito.it}\vspace{1mm}\\
\vspace{1mm}\\
{\bf\large Guglielmo Feltrin}
\vspace{1mm}\\
{\it\small SISSA - International School for Advanced Studies}\\
{\it\small via Bonomea 265}, {\it\small 34136 Trieste, Italy}\\
{\it\small e-mail: guglielmo.feltrin@sissa.it}\vspace{1mm}\\
\vspace{1mm}\\
{\bf\large Fabio Zanolin}
\vspace{1mm}\\
{\it\small Department of Mathematics and Computer Science, University of Udine}\\
{\it\small via delle Scienze 206},
{\it\small 33100 Udine, Italy}\\
{\it\small e-mail: fabio.zanolin@uniud.it}\vspace{1mm}}

\date{}

\maketitle

\vspace{-2mm}

\begin{abstract}
\noindent
We study the periodic and the Neumann boundary value problems associated with the second order nonlinear differential equation
\begin{equation*}
u'' + c u' + \lambda a(t) g(u) = 0,
\end{equation*}
where $g \colon \mathopen{[}0,+\infty\mathclose{[}\to \mathopen{[}0,+\infty\mathclose{[}$ is a sublinear function at infinity
having superlinear growth at zero. We prove the existence of two positive solutions when $\int_{0}^{T} a(t)~\!dt < 0$
and $\lambda > 0$ is sufficiently large. Our approach is based on Mawhin's coincidence degree theory and index computations.
\blfootnote{\textit{2010 Mathematics Subject Classification:} 34B18, 34B15, 34C25, 47H11.}
\blfootnote{\textit{Keywords:} boundary value problems, positive solutions, indefinite weight, multiplicity results, coincidence degree.}
\end{abstract}

\section{Introduction}\label{section-1}

This paper deals with the periodic boundary value problem associated with the nonlinear second order
ordinary differential equation
\begin{equation}\label{eq-1.1}
u'' + c u' + \lambda a(t) g(u) = 0.
\end{equation}
Let ${\mathbb{R}}^{+}:= \mathopen{[}0,+\infty\mathclose{[}$ denote the set of non-negative real numbers. We suppose that
$a \colon {\mathbb{R}}\to {\mathbb{R}}$ is a locally integrable $T$-periodic function and
$g \colon {\mathbb{R}}^{+} \to {\mathbb{R}}^{+}$ is continuous and such that
\begin{equation*}
g(0) = 0, \qquad g(s) > 0 \quad \text{for } \; s > 0.
\leqno{(g_{*})}
\end{equation*}
The real constant $c$ is arbitrary and results will be given depending on the parameter $\lambda>0$.

We are interested in the search of \textit{positive and $T$-periodic} solutions to \eqref{eq-1.1}, namely we look for $u(t)$
satisfying \eqref{eq-1.1} in the Carath\'{e}odory sense (see \cite{Ha-80}) and such that $u(t+T)=u(t)>0$ for all $t\in {\mathbb{R}}$.

As main assumptions on the nonlinearity we require that $g(s)$ tends to zero for $s\to 0^{+}$ faster than linearly and it has a
sublinear growth at infinity, that is
\begin{equation*}
\lim_{s\to 0^{+}} \dfrac{g(s)}{s} = 0
\leqno{(g_{0})}
\end{equation*}
and
\begin{equation*}
\lim_{s\to +\infty} \dfrac{g(s)}{s} = 0.
\leqno{(g_{\infty})}
\end{equation*}
Under the above hypotheses, the search of positive solutions of \eqref{eq-1.1} satisfying the two-point boundary condition
$u(0) = u(T) = 0$ has been widely studied. Note that in this case its is not restrictive to suppose $c=0$, since one can
always reduce the problem to this situation via a standard change of variables.
Typical theorems guarantee the existence of at least two (positive) solutions
when $a(t)\geq 0$ for all $t$ and $\lambda > 0$ is sufficiently large (cf.~\cite{ErHuWa-94}).
These proofs have been obtained by different techniques,
such as the theory of fixed points for positive operators or critical point theory.
Under additional technical assumptions similar results can be given for the Dirichlet problem
\begin{equation*}
\begin{cases}
\, -\Delta \,u = \lambda \, a(x) \, g(u) & \text{ in } \Omega \\
\, u = 0 & \text{ on } \partial\Omega
\end{cases}
\end{equation*}
as well (see, for instance, \cite{Am-72, Li-82, Ra-7374}).
In the recent paper \cite{BoZa-13} a dynamical system approach has been proposed in order to obtain pairs of positive solutions,
also when $a(t)$ is allowed to change its sign.

Concerning the periodic boundary value problem, analogous results on pairs of positive solutions have
been provided in \cite{GrKoWa-08} for equations of the form
\begin{equation*}
u'' - k u + \lambda a(t)g(u) = 0,
\end{equation*}
with $k > 0$. However, less results seem to be available when $k=0$. One of the peculiar aspects of the periodic BVP
associated with \eqref{eq-1.1} is the fact that the differential operator has a nontrivial kernel (which is made by the constant functions).
A second feature to take into account concerns the fact that we have to impose additional conditions on the weight function.
Indeed, if $u(t) > 0$ is a $T$-periodic solution of \eqref{eq-1.1}, then (after integrating the equation on $\mathopen{[}0,T\mathclose{]}$) one has
that $\int_{0}^{T} a(t) g(u(t))~\!dt =0$, with $g(u(t)) > 0$ for every $t$. Hence $a(t)$ cannot be of constant sign.
These two facts make it unclear how to apply the methods based on the theory of positive operators for cones in Banach spaces.

A first contribution in the periodic problem for \eqref{eq-1.1} was obtained in \cite{BoZa-12} in the case $c=0$. More precisely,
taking advantage of the variational (Hamiltonian) structure of the equation
\begin{equation}\label{eq-1.2}
u'' + \lambda a(t) g(u) = 0,
\end{equation}
critical point theory for the action functional
\begin{equation*}
J_{\lambda}(u) := \int_{0}^{T} \biggl{[}\dfrac{1}{2} (u')^{2} - \lambda a(t) G(u)\biggr{]}~\!dt
\end{equation*}
was used to prove the existence of at least two positive $T$-periodic solutions for \eqref{eq-1.2}, with $\lambda$
positive and large, by assuming $a^{+}\not\equiv 0$ on some interval and
\begin{equation*}
\int_{0}^{T} a(t)~\!dt < 0.
\leqno{(a_{*})}
\end{equation*}
Roughly speaking, condition $(a_{*})$ guarantees both that the functional $J_{\lambda}$ is coercive and bounded from below and
that the origin is a strict local minimum. When $\lambda > 0$ is sufficiently large (so that $\inf J_{\lambda} < 0$) one gets two
nontrivial critical points: a global minimum and a second one from a mountain pass geometry. To perform the technical estimates, in
\cite{BoZa-12} some further conditions on $g(s)$ and $G(s) := \int_{0}^{s} g(\xi)~\!d\xi$ (implying $(g_{0})$ and
$(g_{\infty})$) were imposed. For example, the superlinearity assumption at zero is expressed by
\begin{equation*}
\lim_{s\to 0^{+}} \dfrac{g(s)}{s^{\alpha}} = \ell_{\alpha} > 0,
\leqno{(g_{\alpha})}
\end{equation*}
for some $\alpha > 1$. Notice that assumptions of this kind have been used also in previous works dealing with indefinite
superlinear problems, like \cite{AlTa-93, BeCaDoNi-94}.

As observed in \cite{BoZa-12} (and first also in \cite{BaPoTe-88}, in the context of the Neumann BVP), condition $(a_{*})$
becomes necessary when $g(s)$ is continuously differentiable with $g'(s) > 0$ for all $s > 0$. Repeating the same argument
as in \cite[Proposition~2.1]{BoZa-12}, one can check that the same necessary condition is valid for \eqref{eq-1.1} with
an arbitrary $c\in{\mathbb{R}}$.

Unlike the case of the two-point (Dirichlet) boundary value problem, where it is easy to enter in a variational formulation
of Sturm-Liouville type for an arbitrary $c\in\mathbb{R}$, for the periodic problem this is no more guaranteed.
Indeed, for $c\neq0$, we lose the Hamiltonian structure if we pass to the natural equivalent system in the phase-plane
\begin{equation*}
u' = y, \qquad y' = - cy - \lambda a(t) g(u).
\end{equation*}
On the other hand, we can consider an equivalent first order system of Hamiltonian type, as
\begin{equation*}
u' = e^{-ct} y, \qquad y' = - \lambda e^{ct} a(t) g(u),
\end{equation*}
but its $T$-periodic solutions do not correspond to the $T$-periodic solutions of \eqref{eq-1.1}.

\medskip

The main contribution of the present paper is to provide an existence result for pairs of positive $T$-periodic solutions
to equation \eqref{eq-1.1} in the possibly non-variational setting (when $c\neq0$).
To this aim, we introduce a topological approach which may have some independent interest, even for the case $c=0$.
Our proof is reminiscent of the classical approach in the case of positive operators in ordered Banach spaces which
consists in proving that the fixed point index of the associated operator is $1$ on small balls $B(0,r)$
as well as on large balls $B(0,R)$.
Moreover, when $\lambda > 0$ is sufficiently large, one can find an intermediate ball $B(0,\rho)$ (with
$r < \rho < R$) where the fixed point index is $0$. In this manner,
there is a nontrivial (positive) solution
in $P \cap (B(0,\rho)\setminus B[0,r])$ and another one in $P \cap (B(0,R)\setminus B[0,\rho])$, where
$P$ is the positive cone.
In our setting we do not have a positive operator, but, using a maximum principle type argument,
we can work directly with the topological degree in the Banach space of continuous $T$-periodic functions
and then prove that the two nontrivial solutions that we reach are indeed positive. Actually, the situation is even more
complicated because equation \eqref{eq-1.1} is a \textit{coincidence equation} of the form
\begin{equation*}
Lu = N_{\lambda}u,
\end{equation*}
with $L$ a non-invertible differential operator. In this case Mawhin's coincidence degree theory (see \cite{Ma-79}), adapted to the case of
locally compact operators (cf.~\cite{Nu-93}), is the appropriate tool for our purposes.
In the recent paper \cite{FeZa-14} a similar approach has been adopted for the study of positive solutions when the nonlinearity
is superlinear both at zero and at infinity. In such a situation the existence of at least one positive solution is guaranteed.

The advantage of using an approach based on degree theory lies also on the fact that the existence results are stable
with respect to small perturbations of the differential equation. Hence, we can provide pairs of positive $T$-periodic
solutions also for equations of the form
\begin{equation*}
u'' + c u' + \varepsilon u + \lambda a(t) g(u) = 0,
\end{equation*}
for $\varepsilon$ small. This gives an interesting result also in the variational case (when $c=0$).

The technical assumptions on $g(s)$ that we have to impose at zero (as well as at infinity) allow to
slightly improve $(g_{\alpha})$, by using a condition of \textit{regular oscillation} type.
Let ${\mathbb{R}}^{+}_{0}:= \mathopen{]}0,+\infty\mathclose{[}$ and let $h \colon {\mathbb{R}}^{+}_{0} \to {\mathbb{R}}^{+}_{0}$ be a continuous function.
We say that $h$ is \textit{regularly oscillating at zero} if
\begin{equation*}
\lim_{\substack{s\to0^{+} \\ \omega\to1}}\dfrac{h(\omega s)}{h(s)}=1.
\end{equation*}
Analogously, we say that $h$ is \textit{regularly oscillating at infinity} if
\begin{equation*}
\lim_{\substack{s\to+\infty \\ \omega\to1}}\dfrac{h(\omega s)}{h(s)}=1.
\end{equation*}
The concept of regularly oscillating function (usually referred to the case at infinity) is related to
classical conditions of Karamata type which have been developed and studied by several authors for their
significance in different areas of real analysis and probability (cf.~\cite{BiGoTe-87, Se-76}). For the specific definition
considered in our paper as well as for some historical remarks, see \cite{DjTo-01} and the references therein.
Observe that any function $h(s)$ such that $h(s) \sim K s^{p}$, with $K, p > 0$, is regularly oscillating both
at zero and at infinity. However, the class of regularly oscillating functions is quite broad. For instance,
functions like
\begin{equation*}
h(s) = s^{p} \exp\biggl{(}\int_{1}^{s} \dfrac{b(t)}{t}~\!dt\biggr{)},
\end{equation*}
with $b(t)$ continuous and bounded, are regularly oscillating at infinity.

\medskip

Now we are in position to state our main result.

\begin{theorem}\label{th-1.1}
Let $g \colon {\mathbb{R}}^{+} \to {\mathbb{R}}^{+}$ be a continuous function satisfying $(g_{*})$.
Suppose also that $g$ is regularly oscillating at zero and at infinity and satisfies $(g_{0})$ and $(g_{\infty})$.
Let $a \colon \mathbb{R} \to \mathbb{R}$ be a locally integrable $T$-periodic function satisfying the average
condition $(a_{*})$.
Furthermore, suppose that there exists an interval $I \subseteq \mathopen{[}0,T\mathclose{]}$ such that
$a(t) \geq 0$ for a.e.~$t\in I$ and $\int_{I} a(t)~\!dt > 0$.
Then there exists $\lambda^{*}>0$ such that for each $\lambda > \lambda^{*}$ equation \eqref{eq-1.1} has at least two
positive $T$-periodic solutions.
\end{theorem}

As will become clear from the proof, the constant $\lambda^{*}$ can be chosen depending (besides on $c$ and $g(s)$) only on
the behavior of $a(t)$ on the interval $I$. This remark allows to obtain the following corollary
for the related two-parameter equation
\begin{equation}\label{eq-1.4}
u'' + c u' + (\lambda a^{+}(t) - \mu a^{-}(t))g(u) = 0,
\end{equation}
with $\lambda, \mu > 0$, where, as usual, we have set
\begin{equation*}
a^{+}(t):= \dfrac{a(t)+|a(t)|}{2},  \qquad  a^{-}(t):= \dfrac{-a(t)+|a(t)|}{2}.
\end{equation*}
Equation \eqref{eq-1.4}, for $c=0$, has been considered in \cite{BoZa-12b}, with the aim of
investigating multiplicity results and complex dynamics when $\mu \gg 0$ (see also \cite{FeZa-PP} and the references
therein for related results in the superlinear case).

\begin{corollary}\label{cor-1.1}
Let $g(s)$ be as above and let $a(t)$ be a $T$-periodic function with $a^{\pm}\in L^{1}(\mathopen{[}0,T\mathclose{]})$
and $a^-\not\equiv 0$.
Suppose also that there exists an interval $I\subseteq \mathopen{[}0,T\mathclose{]}$ such that
\begin{equation*}
\int_{I} a^{-}(t)~\!dt = 0 < \int_{I} a^{+}(t)~\!dt.
\end{equation*}
Then there exists $\lambda^{*} >0$ such that for each $\lambda > \lambda^{*}$ and for each
\begin{equation*}
\mu > \lambda \, \dfrac{\int_{0}^{T} a^{+}(t)~\!dt}{\int_{0}^{T} a^{-}(t)~\!dt}
\end{equation*}
equation \eqref{eq-1.4} has at least two positive $T$-periodic solutions.
\end{corollary}

Our results are sharp in the sense that there are examples of functions $g(s)$
satisfying all the assumptions of Theorem~\ref{th-1.1} or of Corollary~\ref{cor-1.1} and such that
there are no positive $T$-periodic solutions if $\lambda > 0$ is small or if $(a_{*})$ is not satisfied.
For this remark see \cite[Section~2]{BoZa-12}, where the assertions were proved in the case $c=0$.
One can easily check that those results can be extended to the case of an arbitrary $c\in {\mathbb{R}}$
(see also Section~\ref{section-4.3}).

Another sharp result can be given when $g(s)$ is smooth. Indeed, first of all we produce a variant
of Theorem~\ref{th-1.1} by replacing the hypothesis of regular oscillation of $g$ at zero
or at infinity with the condition of continuous differentiability of $g(s)$ in a neighborhood of $s=0$ or,
respectively, near infinity (see Theorem~\ref{th-4.1}).
Next, in the smooth case and further assuming that $|g'(s)|$ is bounded on $\mathbb{R}^{+}_{0}$, we can also provide a nonexistence result
for $\lambda > 0$ small (see Theorem~\ref{th-4.2}). As a consequence of these results, the following variant of Theorem~\ref{th-1.1} can be stated.
We denote by $g'(\infty) = \lim_{s\to+\infty} g'(s)$.

\begin{theorem}\label{th-1.2}
Let $g \colon {\mathbb{R}}^{+} \to {\mathbb{R}}^{+}$ be a continuously differentiable function satisfying $(g_{*})$
and such that $g'(0) = 0$ and $g'(\infty) = 0$.
Let $a \colon \mathbb{R} \to \mathbb{R}$ be a locally integrable $T$-periodic function satisfying the average
condition $(a_{*})$.
Furthermore, suppose that there exists an interval $I \subseteq \mathopen{[}0,T\mathclose{]}$ such that
$a(t) \geq 0$ for a.e.~$t\in I$ and $\int_{I} a(t)~\!dt > 0$.
Then there exists $\lambda_{*} >0$ such that for each $0 < \lambda < \lambda_{*}$
equation \eqref{eq-1.1} has no positive $T$-periodic solution. Moreover, there exists
$\lambda^{*} >0$ such that for each $\lambda > \lambda^{*}$ equation \eqref{eq-1.1} has at least two positive $T$-periodic solutions.
Condition $(a_{*})$ is also necessary if $g'(s) > 0$ for $s > 0$.
\end{theorem}

To show a simple example of applicability of Theorem~\ref{th-1.2}, we consider the $T$-periodic boundary value problem
\begin{equation}\label{eq-1.5}
\begin{cases}
\, u'' + c u' + \lambda(\sin(t) + k) g(u) = 0\\
\, u(2\pi) - u(0) = u'(2\pi) - u'(0) = 0,
\end{cases}
\end{equation}
where $k\in\mathbb{R}$ and
\begin{equation*}
g(s) = \arctan(s^{\alpha}), \quad \text{with }\, \alpha > 1,
\end{equation*}
(other examples of functions $g(s)$ can be easily produced). Since $g'(s) > 0$ for all $s > 0$,
we know that there are positive $T$-periodic solutions only if $-1 < k < 0$.
Moreover, for any fixed $k\in \mathopen{]}-1,0\mathclose{[}$ there exist
two constants $0 < \lambda_{*,k} \leq \lambda^{*,k}$ such that for $0 < \lambda <  \lambda_{*,k}$ there are no positive solutions for
problem \eqref{eq-1.5}, while for $\lambda >  \lambda^{*,k}$ there are at least two positive
solutions. Estimates for $\lambda_{*,k}$ and $\lambda^{*,k}$ can be given for any specific equation.

\medskip

The plan of the paper is the following. In Section~\ref{section-2} we recall some basic facts about Mawhin's coincidence degree
and we present two lemmas for the computation of the degree (see Lemma~\ref{lem-2.1-deg0} and Lemma~\ref{lem-2.2-deg1}).
We end the section by showing the general scheme we follow in the proof of Theorem~\ref{th-1.1}, which is performed in
Section~\ref{section-3}. We present in Section~\ref{section-4} some consequences and variants of the main theorem
(including existence of small/large solutions using only conditions for $g(s)$ near zero/near infinity, respectively).
In the same section we also deal with the smooth case and give a nonexistence result.
Section~\ref{section-5} is devoted to a brief description of how all the results can be
adapted to the Neumann problem, including a final application to radially symmetric solutions on annular domains.

\section{The abstract setting}\label{section-2}

Let $X:=\mathcal{C}_{T}$ be the Banach space of continuous and $T$-periodic functions $u \colon \mathbb{R} \to \mathbb{R}$,
endowed with the norm
\begin{equation*}
\|u\|_{\infty} := \max_{t\in \mathopen{[}0,T\mathclose{]}} |u(t)| = \max_{t\in \mathbb{R}} |u(t)|,
\end{equation*}
and let $Z:=L^{1}_{T}$ be the Banach space of measurable and $T$-periodic functions $v \colon \mathbb{R} \to \mathbb{R}$
which are integrable on $\mathopen{[}0,T\mathclose{]}$, endowed with the norm
\begin{equation*}
\|v\|_{L^{1}_{T}}:= \int_{0}^{T} |v(t)|~\!dt.
\end{equation*}
The linear differential operator
\begin{equation*}
L \colon u \mapsto - u'' - cu'
\end{equation*}
is a (linear) Fredholm map of index zero defined on $\text{\rm dom}\,L := W^{2,1}_{T} \subseteq X$, with range
\begin{equation*}
\text{\rm Im}\,L =  \biggl{\{}v\in Z \colon \int_{0}^{T} v(t)~\!dt = 0 \biggr\}.
\end{equation*}
Associated with $L$ we have the projectors
\begin{equation*}
P \colon X \to \ker L \cong {\mathbb{R}}, \qquad Q \colon Z \to \text{\rm coker}\,L \cong Z/\text{\rm Im}\,L \cong \mathbb{R},
\end{equation*}
that, in our situation, can be chosen as the average operators
\begin{equation*}
Pu = Qu := \dfrac{1}{T}\int_{0}^{T} u(t)~\!dt.
\end{equation*}
Finally, let
\begin{equation*}
K_{P} \colon \text{\rm Im}\,L \to \text{\rm dom}\,L \cap \ker P
\end{equation*}
be the right inverse of $L$, which is the operator that at any function $v\in L^{1}_{T}$ with $\int_{0}^{T} v(t)~\!dt =0$
associates the unique $T$-periodic solution $u$ of
\begin{equation*}
u'' + c u' + v(t) =0, \quad \text{ with }\; \int_{0}^{T} u(t)~\!dt = 0.
\end{equation*}

Next, we define the $L^{1}$-Carath\'{e}odory function
\begin{equation*}
f_{\lambda}(t,s) :=
\begin{cases}
\, -s,  & \text{if } s \leq 0;\\
\, \lambda a(t) g(s), & \text{if } s \geq 0;
\end{cases}
\end{equation*}
where $a \colon \mathbb{R} \to \mathbb{R}$ is a $T$-periodic and locally integrable function, $g \colon {\mathbb{R}}^{+}\to {\mathbb{R}}^{+}$
is a continuous function with $g(0) = 0$ and $\lambda > 0$ is a fixed parameter.
Let us denote by $N_{\lambda} \colon X \to Z$ the Nemytskii operator induced by the function $f_{\lambda}$, that is
\begin{equation*}
(N_{\lambda} u)(t):= f_{\lambda}(t,u(t)), \quad t\in\mathbb{R}.
\end{equation*}
By coincidence degree theory we know that the equation
\begin{equation}\label{eq-coincidence-eq}
Lu = N_{\lambda}u,\quad u\in \text{\rm dom}\,L,
\end{equation}
is equivalent to the fixed point problem
\begin{equation*}
u = \Phi_{\lambda}u:= Pu + QN_{\lambda}u + K_{P}(Id-Q)N_{\lambda}u, \quad u \in X.
\end{equation*}
Technically, the term $QN_{\lambda}u$ in the above formula should be more correctly written as $JQN_{\lambda}u$, where $J$ is a
linear (orientation-preserving) isomorphism from $\text{\rm coker}\,L$ to $\ker L$. However, in our situation, we can take as $J$ the identity on
$\mathbb{R}$, having identified $\text{\rm coker}\,L$, as well as $\ker L$, with $\mathbb{R}$.
It is standard to verify that $\Phi_{\lambda} \colon X \to X$ is a completely continuous operator.
In such a situation, we usually say that $N_{\lambda}$ is \textit{$L$-completely continuous}
(see \cite{Ma-79}, where the treatment has been given for the most general cases).

If ${\mathcal{O}}\subseteq X$ is an open and \textit{bounded} set such that
\begin{equation*}
Lu \neq N_{\lambda}u, \quad \forall \, u\in \partial{\mathcal{O}}\cap \text{\rm dom}\,L,
\end{equation*}
the \textit{coincidence degree} $D_{L}(L-N_{\lambda},{\mathcal{O}})$ (\textit{of $L$ and $N_{\lambda}$ in ${\mathcal{O}}$})
is defined as
\begin{equation*}
D_{L}(L-N_{\lambda},{\mathcal{O}}):= \text{deg}_{LS}(Id - \Phi_{\lambda},{\mathcal{O}},0),
\end{equation*}
where ``$\text{deg}_{LS}$'' denotes the Leray-Schauder degree.

In our applications we need to consider a slight extension of coincidence degree to open (not necessarily bounded) sets.
To this purpose, we just follow the standard approach used to define the Leray-Schauder degree
for locally compact maps defined on open sets, which is classical in the theory of fixed point index
(cf.~\cite{Gr-72,Ma-99,Nu-85,Nu-93}). More in detail,
let $\Omega\subseteq X$ be an open set and suppose that the solution set
\begin{equation*}
\text{\rm Fix}\,(\Phi_{\lambda},\Omega):= \bigl{\{}u\in {\Omega} \colon u =
\Phi_{\lambda}u\bigr{\}} = \bigl{\{}u\in {\Omega}\cap \text{\rm dom}\,L \colon Lu = N_{\lambda}u\bigr{\}}
\end{equation*}
is compact. The extension of the Leray-Schauder degree to open (not necessarily
bounded) sets allows to define
\begin{equation*}
\text{deg}_{LS}(Id - \Phi_{\lambda},\Omega,0):=\text{deg}_{LS}(Id - \Phi_{\lambda},\mathcal{V},0),
\end{equation*}
where $\mathcal{V}$ is an open and bounded set with
\begin{equation}\label{eq-2.1}
\text{\rm Fix}\,(\Phi_{\lambda},\Omega) \subseteq \mathcal{V} \subseteq \overline{\mathcal{V}} \subseteq \Omega.
\end{equation}
One can check that the definition is independent of the choice of $\mathcal{V}$.
Accordingly, we define the \textit{coincidence degree} $D_{L}(L-N_{\lambda},\Omega)$
(\textit{of $L$ and $N_{\lambda}$ in $\Omega$}) as
\begin{equation*}
D_{L}(L-N_{\lambda},\Omega):= D_{L}(L-N_{\lambda},{\mathcal{V}}) =\text{deg}_{LS}(Id - \Phi_{\lambda},\mathcal{V},0),
\end{equation*}
with $\mathcal{V}$ as above.
In the special case when  $\Omega$ is an open and \textit{bounded} set such that
\begin{equation}\label{eq-2.2}
Lu \neq N_{\lambda}u, \quad \forall \, u\in \partial\Omega\cap \text{\rm dom}\,L,
\end{equation}
it is easy to verify that the above definition is exactly the usual definition of
coincidence degree, according to Mawhin. Indeed, if \eqref{eq-2.2} holds with $\Omega$ open and bounded, then,
by the excision property of the Leray-Schauder degree, we have
$\text{deg}_{LS}(Id - \Phi_{\lambda},\mathcal{V},0) = \text{deg}_{LS}(Id - \Phi_{\lambda},\Omega,0)$
for each open and bounded set ${\mathcal V}$ satisfying \eqref{eq-2.1}.
We refer also to \cite{Mo-96} for an analogous introduction from a
different point of view.

Combining the properties of coincidence degree from \cite[Chapter~II]{Ma-79}
with the theory of fixed point index for locally compact operators (cf.~\cite{Nu-85, Nu-93}), it is possible to derive the following versions
of the main properties of the degree.

\begin{itemize}
\item \textit{Additivity. }
Let $\Omega_{1}$, $\Omega_{2}$ be open and disjoint subsets of $\Omega$ such that $\text{\rm Fix}\,(\Phi_{\lambda},\Omega)\subseteq \Omega_{1}\cup\Omega_{2}$.
Then
\begin{equation*}
D_{L}(L-N_{\lambda},\Omega) = D_{L}(L-N_{\lambda},\Omega_{1})+ D_{L}(L-N_{\lambda},\Omega_{2}).
\end{equation*}

\item \textit{Excision. }
Let $\Omega_{0}$ be an open subset of $\Omega$ such that $\text{\rm Fix}\,(\Phi_{\lambda},\Omega)\subseteq \Omega_{0}$.
Then
\begin{equation*}
D_{L}(L-N_{\lambda},\Omega)=D_{L}(L-N_{\lambda},\Omega_{0}).
\end{equation*}

\item \textit{Existence theorem. }
If $D_{L}(L-N_{\lambda},\Omega)\neq0$, then $\text{\rm Fix}\,(\Phi_{\lambda},\Omega)\neq\emptyset$,
hence there exists $u\in {\Omega}\cap \text{\rm dom}\,L$ such that $Lu = N_{\lambda}u$.

\item \textit{Homotopic invariance. }
Let $H\colon\mathopen{[}0,1\mathclose{]}\times \Omega \to X$, $H_{\vartheta}(u) := H(\vartheta,u)$, be a continuous homotopy such that
\begin{equation*}
{\mathcal S}:=\bigcup_{\vartheta\in\mathopen{[}0,1\mathclose{]}} \bigl{\{}u\in \Omega\cap \text{\rm dom}\,L \colon Lu=H_{\vartheta}u\bigr{\}}
\end{equation*}
is a compact set and there exists an open neighborhood $\mathcal{W}$ of $\Sigma$ such that $\overline{\mathcal{W}}\subseteq \Omega$ and
$(K_{P}(Id-Q)H)|_{\mathopen{[}0,1\mathclose{]}\times\overline{\mathcal{W}}}$ is a compact map.
Then the map $\vartheta\mapsto D_{L}(L-H_{\vartheta},\Omega)$ is constant on $\mathopen{[}0,1\mathclose{]}$.

\end{itemize}

For more details, proofs and applications, we refer to \cite{GaMa-77,Ma-79,Ma-93} and the references therein.

In the sequel we will apply this general setting in the following manner. We consider a $L$-completely continuous operator ${\mathcal N}$
and an open (not necessarily bounded) set ${\mathcal A}$ such that
the solution set $\{u\in \overline{{\mathcal A}}\cap \text{\rm dom}\,L\colon Lu = {\mathcal N} u\}$
is compact and  disjoint from $\partial{\mathcal A}$. Therefore $D_L(L-{\mathcal N},{\mathcal A})$ is well defined.
We will proceed analogously when dealing with homotopies.

\subsection{Auxiliary lemmas}\label{section-2.1}

Within the framework introduced above, we present now two auxiliary semi-abstract results which are useful for the
computation of the coincidence degree. In the following, we denote by $B(0,d)$ and by $B[0,d]$
the open and, respectively, the closed ball of center the origin and radius $d>0$ in $X$.
For Lemma~\ref{lem-2.1-deg0} and Lemma~\ref{lem-2.2-deg1} we do not require all the assumptions on $a(t)$
and $g(s)$ stated in Theorem~\ref{th-1.1}.
In this way we hope that the two results may have an independent interest beyond that of providing a proof of Theorem~\ref{th-1.1}.

\begin{lemma}\label{lem-2.1-deg0}
Let $\lambda >0$. Let $g \colon \mathbb{R}^{+} \to \mathbb{R}^{+}$ be a continuous function such that $g(0)=0$.
Suppose $a\in L^{1}_{T}$ with $\int_{0}^{T} a(t)~\!dt < 0$.
Assume that there exists a constant $d > 0$ and a compact interval $\mathcal{I}\subseteq \mathopen{[}0,T\mathclose{]}$ such that the following properties hold.
\begin{itemize}
\item[$(A_{d,\mathcal{I}})$]
If $\alpha \geq 0$, then any non-negative $T$-periodic solution $u(t)$ of
\begin{equation}\label{eq-2.3}
u'' + c u' + \lambda a(t) g(u) + \alpha = 0
\end{equation}
satisfies $\,\max_{t\in \mathcal{I}} u(t) \neq d$.
\item[$(B_{d,\mathcal{I}})$]
For every $\beta \geq 0$ there exists a constant $D_{\beta} \geq d$ such that if $\alpha\in \mathopen{[}0,\beta\mathclose{]}$ and
$u(t)$ is any non-negative $T$-periodic solution of equation \eqref{eq-2.3}
with $\,\max_{t\in \mathcal{I}} u(t) \leq d$, then $\,\max_{t\in \mathopen{[}0,T\mathclose{]}} u(t) \leq D_{\beta}$.
\item[$(C_{d,\mathcal{I}})$]
There exists $\alpha^{*} \geq 0$ such that equation \eqref{eq-2.3}, with $\alpha=\alpha^{*}$, does not possess any
non-negative $T$-periodic solution $u(t)$ with $\,\max_{t\in \mathcal{I}} u(t) \leq d$.
\end{itemize}
Then
\begin{equation*}
D_{L}(L-N_{\lambda},\Omega_{d,\mathcal{I}}) = 0,
\end{equation*}
where
\begin{equation*}
\Omega_{d,\mathcal{I}} := \Bigl{\{} u\in X \colon \max_{t\in \mathcal{I}} |u(t)| < d \Bigr{\}}.
\end{equation*}
\end{lemma}

Notice that $\Omega_{d,\mathcal{I}}$ is open but not bounded (unless $\mathcal{I} = \mathopen{[}0,T\mathclose{]}$).

\begin{proof}
For a fixed constant $d > 0$ and a compact interval $\mathcal{I}\subseteq \mathopen{[}0,T\mathclose{]}$ as in the statement,
let us consider the open set $\Omega_{d,\mathcal{I}}$ defined above.
We study the equation
\begin{equation}\label{eq-2.4}
u'' + c u' + f_{\lambda}(t,u) + \alpha = 0,
\end{equation}
for $\alpha \geq 0$, which can be written as a coincidence equation in the space $X$
\begin{equation*}
Lu = N_{\lambda}u + \alpha {\bf 1}, \quad u\in \text{\rm dom}\,L,
\end{equation*}
where ${\bf 1}\in X$ is the constant function ${\bf 1}(t) \equiv 1$.

As a first step, we check that the coincidence degree $D_{L}(L- N_{\lambda} -\alpha{\bf 1},\Omega_{d,\mathcal{I}})$
is well defined for any $\alpha \geq 0$. To this aim, suppose that $\alpha\geq 0$ is fixed and consider the set
\begin{equation*}
\begin{aligned}
\mathcal{R}_{\alpha}
&:=\bigl{\{}u\in \text{\rm cl}\,(\Omega_{d,\mathcal{I}})\cap \text{\rm dom}\,L  \colon Lu = N_{\lambda}u + \alpha {\bf 1}\bigr{\}}
\\ &\hspace*{2.5pt}= \bigl{\{}u\in \text{\rm cl}\,(\Omega_{d,\mathcal{I}}) \colon u = \Phi_{\lambda}u + \alpha {\bf 1}\bigr{\}}.
\end{aligned}
\end{equation*}
We have that $u\in \mathcal{R}_{\alpha}$ if and only if $u(t)$ is a $T$-periodic solution of \eqref{eq-2.4} such that
$|u(t)|\leq d$ for every $t\in \mathcal{I}$. By a standard application of the maximum principle,
we find that $u(t)\geq 0$ for all $t\in {\mathbb{R}}$ and, indeed, $u(t)$ solves \eqref{eq-2.3}, with
$\max_{t\in \mathcal{I}} u(t) \leq d$. Condition $(B_{d,\mathcal{I}})$ gives a constant $D_{\alpha}$
such that $\|u\|_{\infty}\leq D_{\alpha}$ and so $\mathcal{R}_{\alpha}$ is bounded. The complete continuity of the operator
$\Phi_{\lambda}$ ensures the compactness of $\mathcal{R}_{\alpha}$. Moreover, condition $(A_{d,\mathcal{I}})$
guarantees that $|u(t)| < d$ for all $t\in \mathcal{I}$ and then we conclude that $\mathcal{R}_{\alpha}\subseteq \Omega_{d,\mathcal{I}}$.
In this manner we have proved that the coincidence degree
$D_{L}(L-N_{\lambda} - \alpha {\bf 1},\Omega_{d,\mathcal{I}})$ is well defined for any $\alpha\geq 0$.

Now, condition  $(C_{d,\mathcal{I}})$, together with the property of existence of solutions when the degree $D_{L}$ is nonzero,
implies that there exists $\alpha^{*}\geq 0$ such that
\begin{equation*}
D_{L}(L-N_{\lambda} - \alpha^{*} {\bf 1},\Omega_{d,\mathcal{I}})=0.
\end{equation*}
On the other hand, from condition $(B_{d,\mathcal{I}})$ applied on the interval
$\mathopen{[}0,\beta\mathclose{]} := \mathopen{[}0,\alpha^{*}\mathclose{]}$,
repeating the same argument as in the first step above, we find that the set
\begin{equation*}
\begin{aligned}
\mathcal{S}
:= \bigcup_{\alpha\in \mathopen{[}0,\alpha^{*}\mathclose{]}} \mathcal{R}_{\alpha} \,
&= \bigcup_{\alpha\in \mathopen{[}0,\alpha^{*}\mathclose{]}} \bigl{\{}u\in \text{\rm cl}\,(\Omega_{d,\mathcal{I}})\cap \text{\rm dom}\,L \colon Lu = N_{\lambda}u + \alpha {\bf 1}\bigr{\}}
\\ &= \bigcup_{\alpha\in \mathopen{[}0,\alpha^{*}\mathclose{]}} \bigl{\{}u\in \text{\rm cl}\,(\Omega_{d,\mathcal{I}}) \colon u = \Phi_{\lambda}u + \alpha {\bf 1}\bigr{\}}
\end{aligned}
\end{equation*}
is a compact subset of $\Omega_{d,\mathcal{I}}$. Hence, by the homotopic invariance of the coincidence degree, we have that
\begin{equation*}
D_{L}(L-N_{\lambda},\Omega_{d,\mathcal{I}}) = D_{L}(L-N_{\lambda} - \alpha^{*} {\bf 1},\Omega_{d,\mathcal{I}}) = 0.
\end{equation*}
This concludes the proof.
\end{proof}

\begin{lemma}\label{lem-2.2-deg1}
Let $\lambda >0$. Let $g \colon \mathbb{R}^{+} \to \mathbb{R}^{+}$ be a continuous function such that $g(0)=0$.
Suppose $a\in L^{1}_{T}$ with $\int_{0}^{T} a(t)~\!dt < 0$.
Assume that there exists a constant $d > 0$ such that $g(d) > 0$ and the following property holds.
\begin{itemize}
\item[$(H_{d})$]
If $\vartheta\in \mathopen{]}0,1\mathclose{]}$ and $u(t)$ is any non-negative $T$-periodic solution of
\begin{equation}\label{eq-2.5}
u'' + c u' + \vartheta \lambda a(t) g(u) = 0,
\end{equation}
then $\,\max_{t\in \mathopen{[}0,T\mathclose{]}} u(t) \neq d$.
\end{itemize}
Then
\begin{equation*}
D_{L}(L-N_{\lambda},B(0,d)) = 1.
\end{equation*}
\end{lemma}

\begin{proof}
First of all, we claim that there are no solutions to the parameterized coincidence equation
\begin{equation*}
Lu = \vartheta N_{\lambda}u, \quad u\in \partial B(0,d) \cap \text{\rm dom}\,L, \;\; 0 < \vartheta \leq 1.
\end{equation*}
Indeed, if any such a solution $u$ exists, it is a $T$-periodic solution of
\begin{equation*}
u'' + c u' +\vartheta f_{\lambda}(t,u) = 0,
\end{equation*}
with $\|u\|_{\infty} = d$. By the definition of $f_{\lambda}(t,s)$ and a standard application
of the maximum principle, we easily get that $u(t)\geq 0$ for every $t\in \mathbb{R}$.
Therefore, $u(t)$ is a non-negative $T$-periodic solution of \eqref{eq-2.5} with $\max_{t\in \mathopen{[}0,T\mathclose{]}} u(t)= d$.
This contradicts property $(H_{d})$ and the claim is thus proved.

As a second step, we consider $QN_{\lambda}u$ for $u \in \ker L$. Since $\ker L \cong \mathbb{R}$, we have
\begin{equation*}
QN_{\lambda}u = \dfrac{1}{T}\int_{0}^{T} f_{\lambda}(t,s)~\!dt,\quad\text{ for }\, u\equiv \text{constant} = s\in {\mathbb R}.
\end{equation*}
For notational convenience, we set
\begin{equation*}
f^{\#}_{\lambda}(s) := \dfrac{1}{T}\int_{0}^{T} f_{\lambda}(t,s)~\!dt =
\begin{cases}
\, -s,  & \text{if } s \leq 0;\\ \vspace*{1pt}
\, \lambda \biggl{(}\dfrac{1}{T} \displaystyle \int_{0}^{T} a(t)~\!dt \biggr{)} g(s), & \text{if } s \geq 0.
\end{cases}
\end{equation*}
Note that $s f^{\#}_{\lambda}(s) < 0$ for each $s\neq 0$.
As a consequence, we find that $QN_{\lambda}u \neq0$ for each $u\in \partial B(0,d) \cap \ker L$.

An important result from Mawhin's continuation theorem (see \cite[Theorem~2.4]{Ma-93} and also \cite{Ma-69},
where the result was previously given in the context of the periodic problem for ODEs) guarantees that
\begin{equation*}
D_{L}(L-N_{\lambda},B(0,d)) = d_B(-QN_{\lambda}|_{\ker L}, B(0,d)\cap \ker L,0) =d_B(-f^{\#}_{\lambda}, \mathopen{]}-d,d\mathclose{[},0),
\end{equation*}
where ``$d_B$'' denotes the Brouwer degree. This latter degree is clearly equal to $1$ as
\begin{equation*}
-f^{\#}_{\lambda}(-d) = -d < 0 < \lambda \biggl{(}- \dfrac{1}{T}\int_{0}^{T} a(t)~\!dt \biggr{)} g(d) = -f^{\#}_{\lambda}(d).
\end{equation*}
This concludes the proof.
\end{proof}

\subsection{Proof of Theorem~\ref{th-1.1}: the general strategy}\label{section-2.2}

With the aid of the two lemmas just proved, we can give a proof of Theorem~\ref{th-1.1}, as follows.

We fix a constant $\rho > 0$ and consider, for $\mathcal{I}:= I$, the open set
\begin{equation*}
\Omega_{\rho,I}:= \Bigl{\{}u\in X \colon \max_{t\in I} |u(t)| < \rho \Bigr{\}}.
\end{equation*}
First of all, we show that condition $(A_{\rho,I})$ is satisfied provided that $\lambda > 0$ is sufficiently large, say
$\lambda > \lambda^{*}:= \lambda^{*}_{\rho,I}$. Such lower bound for $\lambda$ does not depend on $\alpha$.
Then, we fix an arbitrary $\lambda > \lambda^{*}$ and show that conditions $(B_{\rho,I})$ and $(C_{\rho,I})$
are satisfied as well. In particular, for $\beta = 0$, we find a constant $D_{0} = D_{0}(\rho,I,\lambda) \geq \rho$ such that
any possible solution of
\begin{equation*}
Lu = N_{\lambda} u, \quad u\in \text{\rm cl}\,(\Omega_{\rho,I}) \cap \text{\rm dom}\,L,
\end{equation*}
satisfies
\begin{equation*}
\|u\|_{\infty} \leq D_{0}.
\end{equation*}
In this manner, we have that
\begin{equation*}
B(0,\rho) \subseteq \Omega_{\rho,I} \quad \text{ and } \quad \text{\rm Fix}\,(\Phi_{\lambda},\Omega_{\rho,I}) \subseteq B(0,R), \;\; \forall \, R > D_{0}.
\end{equation*}
Moreover,
\begin{equation*}
D_{L}(L-N_{\lambda},\Omega_{\rho,I}) = D_{L}(L-N_{\lambda},\Omega_{\rho,I}\cap B(0,R)) = 0, \quad \forall \, R > D_{0}.
\end{equation*}

As a next step, using $(g_{0})$ and the regular oscillation of $g(s)$ at zero, we find a positive constant
$r_{0} < \rho$ such that for each $r\in \mathopen{]}0,r_{0}\mathclose{]}$ the condition $(H_{r})$ (of Lemma~\ref{lem-2.2-deg1}) is satisfied and therefore
\begin{equation*}
D_{L}(L-N_{\lambda},B(0,r)) = 1, \quad \forall \, 0 < r \leq r_{0}.
\end{equation*}
With a similar argument, using $(g_{\infty})$ and the regular oscillation of $g(s)$ at infinity, we find a positive constant
$R_{0} > D_{0}$ such that for each $R \geq R_{0}$ the condition $(H_{R})$ is satisfied too and therefore
\begin{equation*}
D_{L}(L-N_{\lambda},B(0,R)) = 1, \quad \forall \, R \geq R_{0}.
\end{equation*}
By the additivity property of the coincidence degree we obtain
\begin{equation}\label{eq-2.6}
D_{L}\bigl{(}L-N_{\lambda},\Omega_{\rho,I} \setminus B[0,r] \, \bigr{)} = -1, \quad \forall \, 0 < r \leq r_{0},
\end{equation}
and
\begin{equation}\label{eq-2.7}
D_{L}\bigl{(}L-N_{\lambda},B(0,R) \setminus \text{\rm cl}\,(\Omega_{\rho,I}\cap B(0,R_{0}))\,\bigr{)} = 1, \quad \forall \, R > R_{0}.
\end{equation}
Thus, in conclusion, we find a first solution $\underline{u}$ of \eqref{eq-coincidence-eq}
with $\underline{u} \in \Omega_{\rho,I} \setminus B[0,r]$ (using \eqref{eq-2.6} for a fixed $r\in\mathopen{]}0,r_{0}\mathclose{]}$)
and a second solution $\overline{u}$ of \eqref{eq-coincidence-eq} with $\overline{u} \in B(0,R)  \setminus \text{\rm cl}\,(\Omega_{\rho,I}\cap B(0,R_{0}))$
(using \eqref{eq-2.7} for a fixed $R > R_{0}$).
Both $\underline{u}(t)$ and $\overline{u}(t)$ are nontrivial $T$-periodic solutions of
\begin{equation*}
u'' + c u' + f_{\lambda}(t,u) = 0
\end{equation*}
and, by the maximum principle, they are actually \textit{non-negative} solutions of \eqref{eq-1.1}.
Finally, since by condition $(g_{0})$ we know that $a(t) g(s)/s$ is $L^{1}$-bounded in a right neighborhood
of $s = 0$, it is immediate to prove (by an elementary form of the strong maximum principle) that such solutions are in fact \textit{strictly positive}.

\section{Proof of Theorem~\ref{th-1.1}: the technical details}\label{section-3}

In this section we give a proof of Theorem~\ref{th-1.1} by following the steps described in Section~\ref{section-2.2}.
To this aim, it is sufficient to check separately the validity of the assumptions in Lemma~\ref{lem-2.1-deg0},
for $\mathcal{I} := I$ and $d = \rho > 0$ a fixed number,
and the ones in Lemma~\ref{lem-2.2-deg1}, for $d = r > 0$ small ($0 < r \leq r_{0}$) and for $d = R > 0$ large ($R\geq R_{0}$).
Notice that $r_{0}$ and $R_{0}$ are chosen after that $\rho$ and also $\lambda>0$ have been fixed.

Throughout the section, for the sake of simplicity, we suppose the validity of all the assumptions in Theorem~\ref{th-1.1}.
However, from a careful checking of the proofs below, one can see that, for the verification of each single lemma, not all of them are needed.

\subsection{Checking the assumptions of Lemma~\ref{lem-2.1-deg0} for $\lambda $ large}\label{section-3.1}

Let $\rho>0$ be fixed. Let $I:= \mathopen{[}\sigma,\tau\mathclose{]}\subseteq \mathopen{[}0,T\mathclose{]}$
be such that $a(t)\geq 0$ for a.e.~$t\in I$ and $\int_{I} a(t)~\!dt > 0$.
We fix $\varepsilon > 0$ such that for
\begin{equation*}
\sigma_{0}:= \sigma + \varepsilon < \tau - \varepsilon =: \tau_{0}
\end{equation*}
it holds that
\begin{equation*}
\int_{\sigma_{0}}^{\tau_{0}} a(t)~\!dt > 0.
\end{equation*}

Let us consider the non-negative solutions of equation \eqref{eq-2.3} for $t\in I$.
Such an equation takes the form
\begin{equation}\label{eq-3.1}
u'' + c u' + h(t,u) = 0,
\end{equation}
where we have set (for notational convenience)
\begin{equation*}
h(t,s) = h_{\lambda,\alpha}(t,s):= \lambda a(t) g(s) + \alpha,
\end{equation*}
where $\lambda > 0$ and $\alpha \geq 0$.
Note that $h(t,s) \geq 0$ for a.e.~$t\in I$ and for all $s \geq 0$.

Writing equation \eqref{eq-3.1} as
\begin{equation*}
\bigl{(}e^{ct} u'\bigr{)}' + e^{ct} h(t,u) = 0,
\end{equation*}
we find that $(e^{ct} u'(t))' \leq 0$ for almost every $t\in I$, so that the
map $t\mapsto e^{ct} u'(t)$ is non-increasing on $I$.

We split the proof into different steps.

\smallskip
\noindent
\textit{Step 1. A general estimate. }
For every non-negative solution $u(t)$ of \eqref{eq-3.1} the following estimate holds:
\begin{equation}\label{eq-3.2}
|u'(t)| \leq \dfrac{u(t)}{\varepsilon}\, e^{|c|T}, \quad \forall \, t\in \mathopen{[}\sigma_{0},\tau_{0}\mathclose{]}.
\end{equation}
To prove this, let us fix $t\in \mathopen{[}\sigma_{0},\tau_{0}\mathclose{]}$. The result is trivially true if $u'(t) = 0$.
Suppose that $u'(t) > 0$ and consider the function $u(t)$ on the interval $\mathopen{[}\sigma,t\mathclose{]}$.
Since $\xi\mapsto e^{c\xi} u'(\xi)$ is non-increasing on $\mathopen{[}\sigma,t\mathclose{]}$, we have
\begin{equation*}
u'(\xi) \geq u'(t) e^{c(t-\xi)}, \quad \forall \, \xi\in \mathopen{[}\sigma,t\mathclose{]}.
\end{equation*}
Integrating on $\mathopen{[}\sigma,t\mathclose{]}$, we obtain
\begin{equation*}
u(t) \geq u(t) - u(\sigma) \geq u'(t) e^{-|c|(t-\sigma)}(t-\sigma) \geq u'(t) e^{-|c|T}\varepsilon
\end{equation*}
and therefore \eqref{eq-3.2} follows. If $u'(t) < 0$ we obtain the same result, after an integration on $\mathopen{[}t,\tau\mathclose{]}$.
Hence, \eqref{eq-3.2} is proved in any case. Observe that only a condition on the sign of $h(t,s)$ is used and,
therefore, the estimate is valid independently on $\lambda > 0$ and $\alpha \geq 0$.

\smallskip
\noindent
\textit{Step 2. Verification of $(A_{\rho,I})$ for $\lambda > \lambda^{*}$, with $\lambda^{*}$ depending on $\rho$ and $I$ but not on $\alpha$. }
Suppose that $u(t)$ is a non-negative $T$-periodic solution of \eqref{eq-2.3} with
\begin{equation*}
\max_{t\in I} u(t) = \rho.
\end{equation*}
Let $t_{0}\in I$ be such that $u(t_{0}) = \rho$ and observe that $u'(t_{0}) = 0$, if $\sigma < t_{0} < \tau$, while
$u'(t_{0}) \leq 0$, if $t_{0} =\sigma$, and $u'(t_{0}) \geq 0$, if $t_{0} = \tau$.

First of all, we prove the existence of a constant $\delta \in \mathopen{]}0,1\mathclose{[}$ such that
\begin{equation}\label{eq-delta}
\min_{t\in \mathopen{[}\sigma_{0},\tau_{0}\mathclose{]}} u(t)\geq \delta \rho.
\end{equation}
This follows from the estimate \eqref{eq-3.2}. Indeed, if $t_{*}\in \mathopen{[}\sigma_{0},\tau_{0}\mathclose{]}$
is such that $u(t_{*}) = \min_{t\in \mathopen{[}\sigma_{0},\tau_{0}\mathclose{]}} u(t)$, we obtain that
\begin{equation}\label{eq-3.4}
|u'(t_{*})| \leq \dfrac{u(t_{*})}{\varepsilon} \, e^{|c|T}.
\end{equation}
On the other hand, by the monotonicity of the function $t\mapsto e^{ct} u'(t)$ in $\mathopen{[}\sigma,\tau\mathclose{]}$,
\begin{equation}\label{eq-3.5}
u'(\xi) e^{c\xi} \geq u'(t_{*}) e^{ct_{*}}, \quad \forall \, \xi\in \mathopen{[}\sigma,t_{*}\mathclose{]},
\end{equation}
and
\begin{equation}\label{eq-3.6}
u'(\xi) e^{c\xi} \leq u'(t_{*}) e^{ct_{*}}, \quad \forall \, \xi\in \mathopen{[}t_{*},\tau\mathclose{]}.
\end{equation}
From the properties about $u'(t_{0})$ listed above, we deduce that if $t_{0} > t_{*}$,
then $u'(t_{0}) \geq 0$ and, therefore, we must have $u'(t_{*}) \geq 0$.
Similarly, if $t_{0} < t_{*}$, then $u'(t_{0}) \leq 0$ and, therefore, we must have $u'(t_{*}) \leq 0$.
The case in which $t_{*}=t_{0}$ can be handled in a trivial way and we do not consider it.
In this manner, we have that one of the two situations occur:
either
\begin{equation}\label{eq-3.7}
\sigma\leq t_{0} < t_{*}\in \mathopen{[}\sigma_{0},\tau_{0}\mathclose{]},
\quad  u(t_{0}) = \rho, \quad u'(\xi) \leq 0, \; \forall \, \xi\in \mathopen{[}t_{0},t_{*}\mathclose{]},
\end{equation}
or
\begin{equation}\label{eq-3.8}
\tau\geq t_{0} > t_{*}\in \mathopen{[}\sigma_{0},\tau_{0}\mathclose{]},
\quad  u(t_{0}) = \rho, \quad u'(\xi) \geq 0, \; \forall \, \xi\in \mathopen{[}t_{*},t_{0}\mathclose{]}.
\end{equation}
Suppose that \eqref{eq-3.7} holds. In this situation, from \eqref{eq-3.5}
we have
$-u'(\xi) \leq - u'(t_{*}) e^{c(t_{*}-\xi)}$ for all $\xi\in \mathopen{[}t_{0},t_{*}\mathclose{]}$
and thus, integrating on $\mathopen{[}t_{0},t_{*}\mathclose{]}$ and using \eqref{eq-3.4}, we obtain
\begin{equation*}
\rho - u(t_{*}) \leq |u'(t_{*})| \, e^{|c|T}(t_{*}-t_{0}) \leq \dfrac{u(t_{*})}{\varepsilon} \, e^{2|c|T}T.
\end{equation*}
This gives \eqref{eq-delta} for
\begin{equation*}
\delta:= \dfrac{\varepsilon}{\varepsilon + e^{2|c|T} T}.
\end{equation*}
We get exactly the same estimate in case of \eqref{eq-3.8}, by using \eqref{eq-3.6} and then integrating on $\mathopen{[}t_{*},t_{0}\mathclose{]}$.
Observe that the constant $\delta\in\mathopen{]}0,1\mathclose{[}$ does not depend on $\lambda$ and $\alpha$.

Having found the constant $\delta$, we now define
\begin{equation*}
\eta = \eta(\rho):= \min \bigl{\{} g(s) \colon s\in \mathopen{[}\delta\rho,\rho\mathclose{]} \bigr{\}}.
\end{equation*}
Then, integrating equation \eqref{eq-2.3} on $\mathopen{[}\sigma_{0},\tau_{0}\mathclose{]}$ and using \eqref{eq-3.2} (for $t=\sigma_{0}$ and $t= \tau_{0}$), we obtain
\begin{equation*}
\begin{aligned}
\lambda \eta \int_{\sigma_{0}}^{\tau_{0}} a(t)~\!dt
   &\leq \lambda \int_{\sigma_{0}}^{\tau_{0}} a(t) g(u(t))~\!dt
\\ &= u'(\sigma_{0}) - u'(\tau_{0}) + c \bigl{(}u(\sigma_{0}) - u(\tau_{0})\bigr{)} - \alpha \, (\tau_{0}-\sigma_{0})
\\ &\leq 2 \dfrac{\rho}{\varepsilon} e^{|c|T} + 2|c| \rho.
\end{aligned}
\end{equation*}
Now, we define
\begin{equation}\label{eq-lambdastar}
\lambda^{*}:= \dfrac{2 \rho \bigl(\varepsilon |c| + e^{|c|T}\bigl )}{\varepsilon\eta \int_{\sigma_{0}}^{\tau_{0}} a(t)~\!dt}.
\end{equation}
Arguing by contradiction, we immediately conclude that there are no (non-negative) $T$-periodic solutions $u(t)$ of \eqref{eq-2.3} with
$\max_{t\in I} u(t) = \rho$ if $\lambda > \lambda^{*}$. Thus condition $(A_{\rho,I})$ is proved.

\smallskip
\noindent
\textit{Step 3. Verification of $(B_{\rho,I})$. } Let $u(t)$ be any non-negative $T$-periodic solution of \eqref{eq-2.3}
with $\max_{t\in I} u(t) \leq \rho$. Let us fix an instant $\hat{t}\in \mathopen{[}\sigma_{0},\tau_{0}\mathclose{]}$.
By \eqref{eq-3.2}, we know that
\begin{equation*}
|u'(\hat{t})| \leq \dfrac{\rho}{\varepsilon} \, e^{|c|T}.
\end{equation*}
Using the fact that
\begin{equation*}
|h(t,s)| \leq M(t)|s| + N(t),\quad \text{for a.e. } t\in \mathopen{[}0,T\mathclose{]}, \; \forall \, s\in \mathbb{R}, \; \forall \, \alpha\in \mathopen{[}0,\beta\mathclose{]},
\end{equation*}
with suitable $M, N\in L^{1}_{T}$ (depending on $\beta$),
from a standard application of the (generalized) Gronwall's inequality (cf.~\cite{Ha-80}), we find a constant $D_{\beta} = D_{\beta}(\rho,\lambda)$ such that
\begin{equation*}
\max_{t\in \mathopen{[}0,T\mathclose{]}} \bigl{(} |u(t)| + |u'(t)| \bigr{)} \leq D_{\beta}.
\end{equation*}
So condition $(B_{\rho,I})$ is verified.

\smallskip
\noindent
\textit{Step 4. Verification of $(C_{\rho,I})$. }
Let $u(t)$ be an arbitrary non-negative $T$-periodic solution of \eqref{eq-2.3} with $\max_{t\in I} u(t) \leq \rho$.
Integrating \eqref{eq-2.3} on $\mathopen{[}\sigma_{0},\tau_{0}\mathclose{]}$ and using \eqref{eq-3.2} (for $t=\sigma_{0}$ and $t= \tau_{0}$),
we obtain
\begin{equation*}
\begin{aligned}
\alpha \, (\tau_{0} - \sigma_{0})
   &= u'(\sigma_{0}) - u'(\tau_{0}) + c \bigl{(} u(\sigma_{0}) - u(\tau_{0}) \bigr{)} - \lambda \biggl{(}\int_{\sigma_{0}}^{\tau_{0}} a(t)g(u(t))~\!dt \biggr{)}
\\ &\leq 2 \dfrac{\rho}{\varepsilon} e^{|c|T} + 2|c| \rho =: K = K(\rho, \varepsilon).
\end{aligned}
\end{equation*}
This yields a contradiction if $\alpha > 0$ is sufficiently large. Hence $(C_{\rho,I})$ is verified,
taking $\alpha^{*} > K / (\tau_{0}-\sigma_{0})$.

\smallskip
\noindent
In conclusion, all the assumptions of Lemma~\ref{lem-2.1-deg0} have been verified for a fixed $\rho > 0$ and for $\lambda > \lambda^{*}$.
\qed

\begin{remark}\label{rem-3.1}
Notice that, among the assumptions of Theorem~\ref{th-1.1}, in this part of the proof we have used only the following ones:
$g(s) > 0$ for all $s \in \mathopen{]}0,\rho\mathclose{]}$, $\limsup_{s\to +\infty} |g(s)|/s < + \infty$, $a\in L^{1}_{T}$ and
$a(t)\geq 0$ for a.e.~$t\in I$, with $\int_{I} a(t)~\!dt > 0$.
$\hfill\lhd$
\end{remark}

\subsection{Checking the assumptions of Lemma~\ref{lem-2.2-deg1} for $r$ small}\label{section-3.2}

We prove that condition $(H_{d})$ of Lemma~\ref{lem-2.2-deg1} is satisfied for $d= r$ sufficiently small.
Indeed, we claim that there exists $r_{0} > 0$ such that there is no non-negative $T$-periodic solution $u(t)$ of
\eqref{eq-2.5} for some $\vartheta \in \mathopen{]}0,1\mathclose{]}$ with $\|u\|_{\infty} = r\in \mathopen{]}0, r_{0}\mathclose{]}$.
Arguing by contradiction, we suppose that there exists a sequence of $T$-periodic functions $u_{n}(t)$ with
$u_{n}(t) \geq 0$ for all $t\in \mathbb{R}$ and such that
\begin{equation}\label{eq-3small}
u''_{n}(t) + c u'_{n}(t) + \vartheta_{n} \lambda a(t) g(u_{n}(t)) = 0,
\end{equation}
for a.e.~$t\in\mathbb{R}$ with $\vartheta_{n} \in \mathopen{]}0,1\mathclose{]}$,
and also such that $\|u_{n}\|_{\infty} = r_{n} \to 0^{+}$.
Let $t^{*}_{n}\in \mathopen{[}0,T\mathclose{]}$ be such that $u_{n}(t^{*}_{n}) = r_{n}$.

We define
\begin{equation*}
v_{n}(t):= \dfrac{u_{n}(t)}{\|u_{n}\|_{\infty}} = \dfrac{u_{n}(t)}{r_{n}}
\end{equation*}
and observe that \eqref{eq-3small} can be equivalently written as
\begin{equation}\label{eq-v_n}
v''_{n}(t) + c v'_{n}(t) + \vartheta_{n} \lambda a(t) q(u_{n}(t))v_{n}(t) = 0,
\end{equation}
where $q \colon {\mathbb{R}}^{+} \to {\mathbb{R}}^{+}$ is defined as $q(s):=g(s)/s$ for $s > 0$ and $q(0)=0$.
Notice that $q$ is continuous on ${\mathbb{R}}^{+}$ (by $(g_{0})$).
Moreover, $q(u_{n}(t))\to 0$ uniformly in $\mathbb{R}$, as a consequence of $\|u_{n}\|_{\infty}\to 0$.
Multiplying equation \eqref{eq-v_n} by $v_{n}$ and integrating on $\mathopen{[}0,T\mathclose{]}$, we find
\begin{equation*}
\|v'_{n}\|_{L^{2}_{T}}^{2} = \int_{0}^{T}v'_{n}(t)^{2}~\!dt \leq \lambda \|a\|_{L^{1}_{T}} \sup_{t\in \mathopen{[}0,T\mathclose{]}}|q(u_{n}(t))| \to 0, \quad \text{as } \; n\to \infty.
\end{equation*}
As an easy consequence $\|v_{n}-1\|_{\infty}\to 0$, as $n\to \infty$.

Integrating \eqref{eq-3small} on $\mathopen{[}0,T\mathclose{]}$ and using the periodic boundary conditions, we have
\begin{equation*}
0 = \int_{0}^{T} a(t) g(u_{n}(t))~\!dt = \int_{0}^{T} a(t) g(r_{n})~\!dt  + \int_{0}^{T} a(t)\bigl{(}g(r_{n} v_{n}(t)) - g(r_{n})\bigr{)}~\!dt
\end{equation*}
and hence, dividing by $g(r_{n}) > 0$, we obtain
\begin{equation*}
0 < - \int_{0}^{T} a(t)~\!dt \leq  \|a\|_{L^{1}_{T}} \sup_{t\in \mathopen{[}0,T\mathclose{]}}\left|\dfrac{g(r_{n} v_{n}(t))}{g(r_{n})} - 1\right|.
\end{equation*}
Using the fact that $g(s)$ is regularly oscillating at zero and $v_{n}(t) \to 1$ uniformly as $n\to \infty$, we find that the
right-hand side of the above inequality tends to zero and thus we achieve a contradiction.
\qed

\begin{remark}\label{rem-3.2}
Notice that, among the assumptions of Theorem~\ref{th-1.1}, in this part of the proof we have used only the following ones (for verifying $(H_{r})$): $g(s) > 0$ for all $s$ in a right neighborhood of $s=0$,
$g(s)$ regularly oscillating at zero and satisfying $(g_{0})$, $a\in L^{1}_{T}$ with $\int_{0}^{T} a(t)~\!dt < 0$.
$\hfill\lhd$
\end{remark}

\subsection{Checking the assumptions of Lemma~\ref{lem-2.2-deg1} for $R$ large}\label{section-3.3}

We are going to check that condition $(H_{d})$ of Lemma~\ref{lem-2.2-deg1} is satisfied for $d=R$ sufficiently large.
In other words, we claim that there exists $R_{0} > 0$ such that there is no non-negative $T$-periodic solution $u(t)$ of
\eqref{eq-2.5} for some $\vartheta \in \mathopen{]}0,1\mathclose{]}$ with $\|u\|_{\infty} = R \geq R_{0}$.
Arguing by contradiction, we suppose that there exists a sequence of $T$-periodic functions $u_{n}(t)$ with
$u_{n}(t) \geq 0$ for all $t\in \mathbb{R}$ and such that
\begin{equation}\label{eq-3large}
u''_{n}(t) + c u'_{n}(t) + \vartheta_{n} \lambda a(t) g(u_{n}(t)) = 0,
\end{equation}
for a.e.~$t\in\mathbb{R}$ with $\vartheta_{n} \in \mathopen{]}0,1\mathclose{]}$,
and also such that $\|u_{n}\|_{\infty} = R_{n} \to +\infty$.
Let $t^{*}_{n}\in \mathopen{[}0,T\mathclose{]}$ be such that $u_{n}(t^{*}_{n}) = R_{n}$.

First of all, we claim that $u_{n}(t) \to +\infty$ uniformly in $t$ (as $n\to \infty$). Indeed, to be more precise, we have that
$u_{n}(t) \geq R_{n}/2$ for all $t$.
To prove this assertion, let us suppose, by contradiction, that $\min u_{n}(t) < R_{n}/2$.
In this case, we can take a maximal compact interval $\mathopen{[}\alpha_{n},\beta_{n}\mathclose{]}$
containing $t^{*}_{n}$ and such that $u_{n}(t) \geq R_{n}/2$ for all $t\in\mathopen{[}\alpha_{n},\beta_{n}\mathclose{]}$.
By the maximality of the interval, we also have that $u_{n}(\alpha_{n}) = u_{n}(\beta_{n}) = R_{n}/2$
with $u_{n}'(\alpha_{n}) \geq 0 \geq u_{n}'(\beta_{n})$.

We set
\begin{equation*}
w_{n}(t) := u_{n}(t) - \dfrac{R_{n}}{2}
\end{equation*}
and observe that $0 \leq w_{n}(t) \leq R_{n}/2$ for all $t\in \mathopen{[}\alpha_{n},\beta_{n}\mathclose{]}$.
Equation \eqref{eq-3large} reads equivalently as
\begin{equation*}
- w''_{n}(t) - c w'_{n}(t) = \vartheta_{n} \lambda a(t) g(u_{n}(t)).
\end{equation*}
Multiplying this equation by $w_{n}(t)$ and integrating on $\mathopen{[}\alpha_{n},\beta_{n}\mathclose{]}$, we obtain
\begin{equation*}
\int_{\alpha_{n}}^{\beta_{n}} w'_{n}(t)^{2} ~\!dt \leq \lambda \|a\|_{L^{1}_{T}} \dfrac{R_{n}}{2} \sup_{\frac{R_{n}}{2} \leq s \leq R_{n}} |g(s)|.
\end{equation*}
From condition $(g_{\infty})$, for any fixed $\varepsilon > 0$ there exists $L_{\varepsilon} > 0$ such that $|g(s)| \leq \varepsilon s$,
for all $s\geq L_{\varepsilon}$. Thus, for $n$ sufficiently large so that $R_{n} \geq 2 L_{\varepsilon}$, we find
\begin{equation*}
\int_{\alpha_{n}}^{\beta_{n}} w'_{n}(t)^{2} ~\!dt \leq \dfrac{1}{2} \lambda \varepsilon R_{n}^{2} \|a\|_{L^{1}_{T}}.
\end{equation*}
By an elementary form of the Poincar\'{e}-Sobolev inequality, we conclude that
\begin{equation*}
\dfrac{R_{n}^{2}}{4} = \max_{t\in \mathopen{[}\alpha_{n},\beta_{n}\mathclose{]}}|w_{n}(t)|^{2} \leq T \int_{\alpha_{n}}^{\beta_{n}} w'_{n}(t)^{2} ~\!dt
\leq \dfrac{1}{2} \lambda \varepsilon T R_{n}^{2} \|a\|_{L^{1}_{T}}
\end{equation*}
and a contradiction is achieved if we take $\varepsilon$ sufficiently small.

Consider now the auxiliary function
\begin{equation*}
v_{n}(t):= \dfrac{u_{n}(t)}{\|u_{n}\|_{\infty}} = \dfrac{u_{n}(t)}{R_{n}}
\end{equation*}
and divide equation \eqref{eq-3large} by $R_{n}$. In this manner we obtain again \eqref{eq-v_n}.
By $(g_{\infty})$ and the fact that $u_{n}(t) \to +\infty$ uniformly in $t$, we conclude that
$q(u_{n}(t))={g(u_{n}(t))}/{u_{n}(t)} \to 0$ uniformly (as $n\to \infty$).
Hence, we are exactly in the same situation as in the case we have already discussed above in Section~\ref{section-3.2} for $r$ small and we can end the
proof in a similar way. More precisely, $\|v'_{n}\|_{L^{2}_{T}}\to 0$ as $n\to \infty$
(this follows by multiplying equation \eqref{eq-v_n} by $v_{n}(t)$ and integrating on $\mathopen{[}0,T\mathclose{]}$)
so that $\|v_{n}-1\|_{\infty} \to 0$, as $n\to \infty$.
Then, integrating equation \eqref{eq-3large} on $\mathopen{[}0,T\mathclose{]}$ and dividing by $g(R_{n}) > 0$, we obtain
\begin{equation*}
0 < - \int_{0}^{T} a(t)~\!dt \leq  \|a\|_{L^{1}_{T}} \sup_{t\in \mathopen{[}0,T\mathclose{]}}\left|\dfrac{g(R_{n} v_{n}(t))}{g(R_{n})} - 1\right|.
\end{equation*}
Using the fact that $g(s)$ is regularly oscillating at infinity and $v_{n}(t) \to 1$ uniformly as $n\to \infty$, we find that the
right-hand side of the above inequality tends to zero and thus we achieve a contradiction.
\qed

\begin{remark}\label{rem-3.3}
Notice that, among the assumptions of Theorem~\ref{th-1.1}, in this part of the proof we have used only the following ones (for verifying $(H_{R})$):
$g(s) > 0$ for all $s$ in a  neighborhood of infinity,
$g(s)$ regularly oscillating at infinity and satisfying $(g_{\infty})$, $a\in L^{1}_{T}$ with $\int_{0}^{T} a(t)~\!dt < 0$.
$\hfill\lhd$
\end{remark}

\section{Related results}\label{section-4}

In this section we present some consequences and variants obtained from Theorem~\ref{th-1.1}. We also examine
the cases of non-existence of solutions when the parameter $\lambda$ is small.

\subsection{Proof of Corollary~\ref{cor-1.1}}\label{section-4.1}

In order to deduce Corollary~\ref{cor-1.1} from Theorem~\ref{th-1.1}, we stress the fact that
the constant $\lambda^{*} >0$ (defined in \eqref{eq-lambdastar}) is produced along the proof of Lemma~\ref{lem-2.1-deg0}
in dependence of an interval $I \subseteq \mathopen{[}0,T\mathclose{]}$ where $a(t)\geq 0$ and $\int_{I} a(t)~\!dt >0$.
For this step in the proof we do not need any information about the weight function on $\mathopen{[}0,T\mathclose{]}\setminus I$.
As a consequence, when we apply our result to equation \eqref{eq-1.4}, we have that $\lambda^{*}$ can be chosen independently on $\mu$.
On the other hand, for Lemma~\ref{lem-2.2-deg1} with  $r$ small as well as with $R$ large,
we do not need any special condition on $\lambda$ (except that $\lambda$ in \eqref{eq-3small} or in \eqref{eq-3large} is fixed)
and we use only the fact that $\int_{0}^{T} a(t)~\!dt <0$ (without requiring any other information on the sign of $a(t)$).
Accordingly, once that $\lambda > \lambda^{*}$ is fixed, to obtain a pair of positive $T$-periodic solutions we only need to check that
the integral of the weight function on $\mathopen{[}0,T\mathclose{]}$ is negative. For equation \eqref{eq-1.4} this request is equivalent to
\begin{equation*}
\dfrac{\mu}{\lambda} > \dfrac{\int_{0}^{T} a^{+}(t)~\!dt}{\int_{0}^{T} a^{-}(t)~\!dt}.
\end{equation*}
By these remarks, we deduce immediately Corollary~\ref{cor-1.1} from Theorem~\ref{th-1.1}.
\qed

\subsection{Existence of small/large solutions}\label{section-4.b}

Theorem~\ref{th-1.1} guarantees the existence of at least two positive $T$-periodic solutions of \eqref{eq-1.1}.
More in detail, we have found a first solution in $\Omega_{\rho,I} \setminus B[0,r]$
and a second one in $B(0,R) \setminus \text{\rm cl}\,(\Omega_{\rho,I}\cap B(0,R_{0}))$,
verifying that the coincidence degree is nonzero in these sets (see \eqref{eq-2.6} and \eqref{eq-2.7}).
The positivity of both the solutions follows from maximum principle arguments.
A careful reading of the proof (cf.~Section~\ref{section-3}) shows that weaker conditions on $g(s)$ are sufficient to repeat some of the
steps in Section~\ref{section-2.2} in order to prove \eqref{eq-2.6} (or \eqref{eq-2.7})
and thus obtain the existence of a small (or large, respectively) positive $T$-periodic solution of \eqref{eq-1.1}.

More precisely, taking into account Remark~\ref{rem-3.1} and Remark~\ref{rem-3.2}
we can state the following theorem, ensuring the existence of a small positive $T$-periodic solution.

\begin{theorem}\label{th-small}
Let $g \colon {\mathbb{R}}^{+} \to {\mathbb{R}}^{+}$ be a continuous function satisfying $(g_{*})$ and
\begin{equation}\label{ginf}
\limsup_{s\to +\infty}\dfrac{g(s)}{s} < +\infty.
\end{equation}
Suppose also that $g$ is regularly oscillating at zero and satisfies $(g_{0})$.
Let $a \colon {\mathbb{R}} \to {\mathbb{R}}$ be a locally integrable $T$-periodic function satisfying the average
condition $(a_{*})$. Furthermore, suppose that there exists an interval $I\subseteq \mathopen{[}0,T\mathclose{]}$ such that
$a(t) \geq 0$ for a.e.~$t\in I$ and $\int_{I} a(t)~\!dt > 0$.
Then there exists $\lambda^{*}>0$ such that for each $\lambda > \lambda^{*}$ equation \eqref{eq-1.1} has at least one
positive $T$-periodic solution.
\end{theorem}

On the other hand, in view of Remark~\ref{rem-3.1} and Remark~\ref{rem-3.3} we have the following result
giving the existence of a large positive $T$-periodic solution.

\begin{theorem}\label{th-large}
Let $g \colon {\mathbb{R}}^{+} \to {\mathbb{R}}^{+}$ be a continuous function satisfying $(g_{*})$ and
\begin{equation}\label{g00}
\limsup_{s\to 0^{+}}\dfrac{g(s)}{s} < +\infty.
\end{equation}
Suppose also that $g$ is regularly oscillating at infinity and satisfies $(g_{\infty})$.
Let $a \colon {\mathbb{R}} \to {\mathbb{R}}$ be a locally integrable $T$-periodic function satisfying the average
condition $(a_{*})$. Furthermore, suppose that there exists an interval $I\subseteq \mathopen{[}0,T\mathclose{]}$ such that
$a(t) \geq 0$ for a.e.~$t\in I$ and $\int_{I} a(t)~\!dt > 0$.
Then there exists $\lambda^{*}>0$ such that for each $\lambda > \lambda^{*}$ equation \eqref{eq-1.1} has at least one
positive $T$-periodic solution.
\end{theorem}

Notice that the possibility of applying a strong maximum principle (in order to obtain positive solutions)
is ensured by $(g_{0})$ in Theorem~\ref{th-small}, while it follows by \eqref{g00} in Theorem~\ref{th-large}.
The dual condition \eqref{ginf} in Theorem~\ref{th-small} is, on the other hand, needed to apply
Gronwall's inequality (checking the assumptions of Lemma~\ref{lem-2.1-deg0}).

\subsection{Smoothness versus regular oscillation}\label{section-4.2}

It can be observed that the assumptions of regular oscillation of $g(s)$ at zero or, respectively, at infinity can be replaced by
suitable smoothness assumptions. Indeed, we can provide an alternative manner to check the assumptions of Lemma~\ref{lem-2.2-deg1}
for $r$ small or $R$ large, by assuming that $g(s)$ is smooth in a neighborhood of zero or, respectively, in a neighborhood of infinity.
For this purpose, we present some preliminary considerations.

Let $u(t)$ be a positive and $T$-periodic solution of
\begin{equation}\label{eq-4.1}
u'' + c u' + \nu a(t) g(u) = 0,
\end{equation}
where $\nu > 0$ is a given parameter (in the following, we will take $\nu = \lambda$ or $\nu = \vartheta\lambda$).
Suppose that the map $g(s)$ is continuously differentiable on an interval containing the range of $u(t)$. In such a situation, we can perform the
change of variable
\begin{equation}\label{eq-4.2}
z(t):=\dfrac{u'(t)}{\nu g(u(t))}
\end{equation}
and observe that $z(t)$ satisfies
\begin{equation}\label{eq-4.3}
z' + c z = - \nu g'(u(t))z^{2} - a(t).
\end{equation}
The function $z(t)$ is absolutely continuous, $T$-periodic with $\int_{0}^{T} z(t)~\!dt = 0$
and, moreover, there exists a $t^{*}\in \mathopen{[}0,T\mathclose{]}$ such that $z(t^{*}) = 0$.

This change of variables (recently considered also in \cite{BoZa-2013}) is used to provide a nonexistence result
as well as a priori bounds for the solutions. We premise the following result.

\begin{lemma}\label{lem-4.1}
Let $J\subseteq \mathbb{R}$ be an interval.
Let $g \colon J \to {\mathbb{R}}^{+}_{0}$ be a continuously differentiable function with bounded derivative (on $J$).
Let $a \in L^{1}_{T}$ satisfy $(a_{*})$.
Then there exists $\omega_{*}>0$ such that, if
\begin{equation*}
\nu\, \sup_{s\in J}|g'(s)| < \omega_{*},
\end{equation*}
there are no $T$-periodic solutions of \eqref{eq-4.1}
with $u(t)\in J$, for all $t\in\mathbb{R}$.
\end{lemma}

\begin{proof}
For notational convenience, let us set
\begin{equation*}
D:= \sup_{s\in J}|g'(s)|.
\end{equation*}
First of all, we fix a positive constants
$M > e^{|c|T} \|a\|_{L^{1}_{T}}$
and define
\begin{equation*}
\omega_{*} := \min \Biggl{\{} \dfrac{M-e^{|c|T}\|a\|_{L^{1}_{T}}}{M^{2} T e^{|c|T}} , \dfrac{-\int_{0}^{T} a(t)~\!dt}{M^{2} T} \Biggr{\}}.
\end{equation*}
Note that $\omega_{*}$ does not depend on $\nu$, $J$ and $D$.
We shall prove that if
\begin{equation*}
0 < \nu D < \omega_{*}
\end{equation*}
equation \eqref{eq-4.1} has no $T$-periodic solution $u(t)$ with range in $J$.

By contradiction we suppose that $u(t)$ is a solution of \eqref{eq-4.1} with $u(t)\in J$, for all $t\in\mathbb{R}$.
Setting $z(t)$ as in \eqref{eq-4.2}, we claim that
\begin{equation}\label{eqz-ineq}
\|z\|_{\infty} \leq M.
\end{equation}
Indeed, if by contradiction we suppose that \eqref{eqz-ineq} is not true, then using the fact that $z(t)$
vanishes at some point of $\mathopen{[}0,T\mathclose{]}$, we can find a maximal interval $\mathcal{I}$
of the form $\mathopen{[}t^{*},\tau\mathclose{]}$ or $\mathopen{[}\tau,t^{*}\mathclose{]}$
such that $|z(t)| \leq M$ for all $t\in \mathcal{I}$ and $|z(t)| > M$ for some $t\notin \mathcal{I}$.
By the maximality of the interval $\mathcal{I}$, we also know that $|z(\tau)| = M$.
Multiplying equation \eqref{eq-4.3} by $e^{c(t-\tau)}$, we achieve
\begin{equation*}
\bigl{(} z(t) e^{c(t-\tau)} \bigr{)}' = \bigl{(} - \nu g'(u(t))z^{2}(t) - a(t)\bigr{)} e^{c(t-\tau)}.
\end{equation*}
Then, integrating on $\mathcal{I}$ and passing to the absolute value, we obtain
\begin{equation*}
\begin{aligned}
M &= |z(\tau)| = |z(\tau) - z(t^{*}) e^{c(t^{*}-\tau)}| \leq  \biggl{|} \int_{\mathcal{I}}
\nu g '(u(t)) z^{2}(t) ~\!dt \biggr{|} \, e^{|c|T}+ \|a\|_{L^{1}_{T}}e^{|c|T} \\
& \leq \nu D M^{2} T e^{|c|T}+ \|a\|_{L^{1}_{T}} e^{|c|T} < \omega_{*} M^{2} T e^{|c|T}+ \|a\|_{L^{1}_{T}} e^{|c|T} \leq M,
\end{aligned}
\end{equation*}
a contradiction. In this manner, we have verified that \eqref{eqz-ineq} is true.

Now, integrating \eqref{eq-4.3} on $\mathopen{[}0,T\mathclose{]}$ and using \eqref{eqz-ineq}, we reach
\begin{equation*}
0 < - \int_{0}^{T} a(t)~\!dt  = \int_{0}^{T} \nu g '(u(t)) z^{2}(t) ~\!dt < \omega_{*} M^{2} T \leq - \int_{0}^{T} a(t)~\!dt,
\end{equation*}
a contradiction. This concludes the proof.
\end{proof}

The same change of variable is employed to provide the following variant of Theorem~\ref{th-1.1}.

\begin{theorem}\label{th-4.1}
Let $g \colon {\mathbb{R}}^{+} \to {\mathbb{R}}^{+}$ be a continuous function satisfying $(g_{*})$
and such that $g(s)$ is continuously differentiable on a right neighborhood
of $s=0$ and on a neighborhood
of infinity. Suppose also that $(g_{0})$ and
\begin{equation*}
g'(\infty):=\lim_{s\to +\infty}{g'(s)} = 0
\leqno{(g'_{\infty})}
\end{equation*}
hold.
Let $a \colon {\mathbb{R}} \to {\mathbb{R}}$ be a locally integrable $T$-periodic function satisfying the average
condition $(a_{*})$.
Furthermore, suppose that there exists an interval $I\subseteq \mathopen{[}0,T\mathclose{]}$ such that
$a(t) \geq 0$ for a.e.~$t\in I$ and $\int_{I} a(t)~\!dt > 0$.
Then there exists $\lambda^{*}>0$ such that for each $\lambda > \lambda^{*}$ equation \eqref{eq-1.1} has at least two
positive $T$-periodic solutions.
\end{theorem}

\begin{proof}
We follow the scheme described in Section~\ref{section-2.2}.
The verification of the assumptions of Lemma~\ref{lem-2.1-deg0} for $\lambda$ large is exactly the same as in Section~\ref{section-3.1}.
We just describe the changes with respect to Section~\ref{section-3.2} and Section~\ref{section-3.3}.
It is important to emphasize that $\lambda > \lambda^{*}$ is fixed from now on.

\smallskip
\noindent
\textit{Verification of the assumption of Lemma~\ref{lem-2.2-deg1} for $r$ small. }
Let $\mathopen{[}0,\varepsilon_{0}\mathclose{[}$ be a right neighborhood of $0$ where $g$ is continuously differentiable.
We claim that there exists $r_{0}\in \mathopen{]}0,\varepsilon_{0}\mathclose{[}$
such that for all $0<r\leq r_{0}$ and
for all $\vartheta\in\mathopen{]}0,1\mathclose{]}$ there are no non-negative $T$-periodic solutions $u(t)$ of \eqref{eq-2.5}
such that $\|u\|_{\infty}=r$.

First of all, we observe that any non-negative $T$-periodic solutions $u(t)$ of \eqref{eq-2.5},
with $\|u\|_{\infty}=r$, is positive. This follows either by the uniqueness of
the trivial solution (due to the smoothness of $g(s)$ in $\mathopen{[}0,\varepsilon_{0}\mathclose{[}$), or by an elementary form
of the strong maximum principle. Thus we have to prove that there are no $T$-periodic solutions $u(t)$ of \eqref{eq-2.5}
with range in the interval $\mathopen{]}0,r\mathclose{]}$ (for all $0 < r \leq r_{0}$).

We apply Lemma~\ref{lem-4.1} to the present situation with $\nu =\vartheta \lambda$
and $J= \mathopen{]}0,r\mathclose{]}$. There exists a constant $\omega_{*} >0$
(independent on $r$) such that there are no
$T$-periodic solutions with range in $\mathopen{]}0,r\mathclose{]}$ if
\begin{equation*}
\sup_{0 <s \leq r}|g'(s)| = \max_{0\leq s \leq r}|g'(s)| < \dfrac{\omega_{*}}{\lambda}
\end{equation*}
(recall that $0 < \vartheta \leq 1$).
This latter condition is clearly satisfied for every $r\in \mathopen{]}0,r_{0}\mathclose{]}$,
with $r_{0} > 0$ suitably chosen
using the continuity of $g'(s)$ at $s=0^{+}$.

\smallskip
\noindent
\textit{Verification of the assumption of Lemma~\ref{lem-2.2-deg1} for $R$ large. }
Let $\mathopen{]}N,+\infty\mathclose{[}$ be a neighborhood of infinity where $g$ is continuously differentiable.
As in Section~\ref{section-3.3}, we argue by contradiction. Suppose that there exists a sequence of
non-negative $T$-periodic functions $u_{n}(t)$ satisfying \eqref{eq-3large} and such that
$\|u_{n}\|_{\infty} = R_{n}\to +\infty$. By the same argument as previously developed therein, we find
that $u_{n}(t) \geq R_{n}/2$, for all $t\in {\mathbb{R}}$ (for $n$ sufficiently large). Notice that for this part of the proof we
require condition $(g_{\infty})$, but we do not need the hypothesis of regular oscillation at infinity.
Clearly, $(g_{\infty})$ is implied by $(g'_{\infty})$.

For $n$ sufficiently large (such that $R_{n}>2 N$),
we apply Lemma~\ref{lem-4.1} to the present situation with $\nu = \nu_{n}:=\vartheta_{n} \lambda$
and $J= J_{n}:=\mathopen{[}R_{n}/2,R_{n}\mathclose{]}$. There exists a constant $\omega_{*} >0$
(independent on $n$) such that there are no
$T$-periodic solutions with range in $J_{n}$ if
\begin{equation*}
\max_{\frac{R_{n}}{2}\leq s \leq R_{n}}|g'(s)| < \dfrac{\omega_{*}}{\lambda}
\end{equation*}
(recall that $0 < \vartheta_{n} \leq 1$).
This latter condition is clearly satisfied for every $n$ sufficiently large
as a consequence of condition $(g'_{\infty})$. The desired contradiction is thus achieved.
\end{proof}

\begin{remark}\label{rem-4.1}
Clearly one can easily produce two further theorems, by combining the assumptions of regular oscillation at zero
(at infinity) with the smoothness condition at infinity (at zero, respectively).
$\hfill\lhd$
\end{remark}

\subsection{Nonexistence results}\label{section-4.3}

In the proof of Theorem~\ref{th-4.1} we have applied Lemma~\ref{lem-4.1} to intervals of the form $\mathopen{]}0,r\mathclose{]}$
or, respectively, $\mathopen{[}R_{n}/2,R_{n}\mathclose{]}$ in order to check the assumptions of Lemma~\ref{lem-2.2-deg1}. Clearly,
one could apply such a lemma to the whole interval $\mathbb{R}^{+}_{0}$ of positive real numbers. In this manner,
we can easily provide a nonexistence result of positive $T$-periodic solutions to \eqref{eq-1.1} when $g'(s)$ is bounded in
$\mathbb{R}^{+}_{0}$ and $\lambda$ is small. With this respect, the following result holds.

\begin{theorem}\label{th-4.2}
Let $g \colon {\mathbb{R}}^{+} \to {\mathbb{R}}^{+}$ be a  continuously differentiable function
satisfying $(g_{*})$, $(g_{0})$ and $(g'_{\infty})$.
Let $a \in L^{1}_{T}$ satisfy $(a_{*})$.
Then there exists $\lambda_{*}>0$ such that for each $0 < \lambda < \lambda_{*}$
equation \eqref{eq-1.1} has no positive $T$-periodic solution.
\end{theorem}

\begin{proof}
First of all, we observe that $g'$ is bounded on $\mathbb{R}^{+}_{0}$
(since $g(s)$ is continuously differentiable in $\mathbb{R}^{+}$ with $g'(0) = g'(\infty) = 0$). Accordingly, let us set
\begin{equation*}
D:= \max_{s\geq 0} |g'(s)|.
\end{equation*}
We apply now Lemma~\ref{lem-4.1} to equation \eqref{eq-1.1} for
$J=\mathbb{R}^{+}_{0}$. This lemma guarantees the existence of a constant $\omega_{*} > 0$ such that, if
$0 < \lambda < \omega_{*}/D$, \eqref{eq-1.1} has no positive $T$-periodic solution. This ensures the existence of a
suitable constant $\lambda_{*} \geq \omega_{*}/D$, as claimed in the statement of the theorem.
\end{proof}

At this point, Theorem~\ref{th-1.2} of the Introduction is a straightforward consequence of Theorem~\ref{th-4.1} and Theorem~\ref{th-4.2}.

\section{Neumann boundary conditions}\label{section-5}

In this final section we briefly describe how to obtain the preceding results for the Neumann boundary value problem.
For the sake of simplicity, we deal with the case $c=0$. If $c\neq0$, we can write equation \eqref{eq-1.1} as
\begin{equation*}
\bigl{(}u' e^{ct} \bigr{)}' + \lambda \tilde{a}(t)g(u) = 0, \quad
\text{ with }\; \tilde{a}(t):= a(t)e^{ct},
\end{equation*}
and enter in the setting of coincidence degree theory for the linear operator $L \colon u\mapsto - (u' e^{ct})'$.
Accordingly, we consider the BVP
\begin{equation}\label{NeumannBVP}
\begin{cases}
\, u'' + \lambda a(t) g(u) = 0 \\
\, u'(0)=u'(T)=0,
\end{cases}
\end{equation}
where $a \colon \mathopen{[}0,T\mathclose{]} \to \mathbb{R}$ and $g(s)$ satisfy the same conditions as in the previous sections.
In this case, the abstract setting of Section~\ref{section-2} can be reproduced almost verbatim with $X:= \mathcal{C}(\mathopen{[}0,T\mathclose{]})$,
$Z:=L^{1}(\mathopen{[}0,T\mathclose{]})$ and $L \colon u\mapsto - u''$, by taking
\begin{equation*}
\text{\rm dom}\,L:= \bigl{\{}u\in W^{2,1}(\mathopen{[}0,T\mathclose{]}) \colon u'(0) = u'(T) = 0 \bigr{\}}.
\end{equation*}
With the above positions $\ker L \cong {\mathbb{R}}$, $\text{\rm Im}\,L$, as well as the projectors $P$
and $Q$ are exactly the same as in Section~\ref{section-2}. All the results till Section~\ref{section-4} can be now
restated for problem \eqref{NeumannBVP}. In particular, we obtain again Theorem~\ref{th-1.1}, Theorem~\ref{th-4.1} and
Theorem~\ref{th-4.2}, as well as their corollaries for equation \eqref{eq-1.1} (with $c=0$) and the Neumann boundary conditions.

We present now a consequence of these results to the study of a PDE in an annular domain.
In order to simplify the exposition of the next results, we assume the continuity of the weight function.
In this manner, the solutions we find are the ``classical'' ones (at least two times continuously differentiable).

\subsection{Radially symmetric solutions}\label{section-5.1}

Let $\|\cdot\|$ be the Euclidean norm in ${\mathbb{R}}^{N}$ (for $N \geq 2$) and let
\begin{equation*}
\Omega:= B(0,R_{2})\setminus B[0,R_{1}] = \bigl{\{}x\in {\mathbb{R}}^{N} \colon R_{1} < \|x\| < R_{2}\bigr{\}}
\end{equation*}
be an open annular domain, with $0 < R_{1} < R_{2}$.

We deal with the Neumann boundary value problem
\begin{equation}\label{eq-pde-rad}
\begin{cases}
\, -\Delta \,u = \lambda \, q(x)\,g(u) & \text{ in } \Omega \\ \vspace*{2pt}
\, \dfrac{\partial u}{\partial {\bf n}} = 0 & \text{ on } \partial\Omega,
\end{cases}
\end{equation}
where $q \colon \overline{\Omega}\to {\mathbb{R}}$ is a continuous function which is radially symmetric, namely
there exists a continuous scalar function ${\mathcal{Q}} \colon \mathopen{[}R_{1},R_{2}\mathclose{]}\to {\mathbb{R}}$
such that
\begin{equation*}
q(x) = {\mathcal{Q}}(\|x\|), \quad \forall \, x\in \overline{\Omega}.
\end{equation*}
We look for existence/nonexistence and multiplicity of radially symmetric positive solutions of \eqref{eq-pde-rad},
that are classical solutions such that $u(x) > 0$ for all $x\in \Omega$ and also $u(x) = {\mathcal{U}}(\|x\|)$,
where ${\mathcal{U}}$ is a scalar function defined on $\mathopen{[}R_{1},R_{2}\mathclose{]}$.

Accordingly, our study can be reduced to the search of positive solutions
of the Neumann boundary value problem
\begin{equation}\label{eq-rad}
{\mathcal{U}}''(r) + \dfrac{N-1}{r} \, {\mathcal{U}}'(r) + \lambda {\mathcal{Q}}(r) g({\mathcal{U}}(r)) = 0,
\quad {\mathcal{U}}'(R_{1}) = {\mathcal{U}}'(R_{2}) = 0.
\end{equation}
Using the standard change of variable
\begin{equation*}
t = h(r):= \int_{R_{1}}^{r} \xi^{1-N} ~\!d\xi
\end{equation*}
and defining
\begin{equation*}
T:= \int_{R_{1}}^{R_{2}} \xi^{1-N} ~\!d\xi, \quad r(t):= h^{-1}(t) \quad \text{and} \quad v(t)={\mathcal{U}}(r(t)),
\end{equation*}
we transform \eqref{eq-rad} into the equivalent problem
\begin{equation}\label{eq-rad1}
v'' +  \lambda a(t) g(v) = 0, \quad v'(0) = v'(T) = 0,
\end{equation}
with
\begin{equation*}
a(t):= r(t)^{2(N-1)}{\mathcal{Q}}(r(t)).
\end{equation*}
Consequently, the Neumann boundary value problem \eqref{eq-rad1} is of the same form of \eqref{NeumannBVP}
and we can apply the previous results.

Notice that condition $(a_{*})$ reads as
\begin{equation*}
0 > \int_{0}^{T}r(t)^{2(N-1)}{\mathcal{Q}}(r(t))~\!dt = \int_{R_{1}}^{R_{2}}r^{N-1}{\mathcal{Q}}(r)~\!dr.
\end{equation*}
Up to a multiplicative constant, the latter integral is the integral of $q(x)$ on $\Omega$,
using the change of variable formula for radially symmetric functions. Thus, $a(t)$ satisfies $(a_{*})$ if and only if
\begin{equation*}
\int_{\Omega}^{} q(x)~\!dx < 0.
\leqno{(q_{*})}
\end{equation*}

The analogue of Theorem~\ref{th-1.1} for problem \eqref{eq-pde-rad} now becomes the following.

\begin{theorem}\label{th-5.1}
Let $g \colon {\mathbb{R}}^{+} \to {\mathbb{R}}^{+}$ be a continuous function satisfying $(g_{*})$.
Suppose also that $g$ is regularly oscillating at zero and at infinity and satisfies $(g_{0})$ and $(g_{\infty})$.
Let $q(x)$ be a continuous (radial) weight function as above satisfying $(q_{*})$ and such that $q(x_{0}) > 0$ for some $x_{0}\in \Omega$.
Then there exists $\lambda^{*}>0$ such that for each $\lambda > \lambda^{*}$ problem \eqref{eq-pde-rad} has at least two
positive radially symmetric solutions.
\end{theorem}

Similarly, if we replace the regularly oscillating conditions with the smoothness assumptions, from Theorem~\ref{th-4.1} and Theorem~\ref{th-4.2},
we obtain the next result.

\begin{theorem}\label{th-5.2}
Let $g \colon {\mathbb{R}}^{+} \to {\mathbb{R}}^{+}$ be a continuously differentiable function satisfying $(g_{*})$, $(g_{0})$ and
$(g'_{\infty})$.
Let $q(x)$ be a continuous (radial) weight function as above satisfying $(q_{*})$ and such that $q(x_{0}) > 0$ for some $x_{0}\in \Omega$.
Then there exist two positive constant $\lambda_{*} \leq \lambda^{*}$ such that for each $0 < \lambda < \lambda_{*}$
there are no positive radially symmetric solutions for problem \eqref{eq-pde-rad}, while for
each $\lambda > \lambda^{*}$ there exist  at least two positive radially symmetric solutions.
Moreover, if $g'(s) > 0$ for all $s > 0$, then condition $(q_{*})$ is also necessary.
\end{theorem}

\bibliographystyle{elsart-num-sort}
\bibliography{BFZ_biblio_2015}

\bigskip
\begin{flushleft}

{\small{\it Preprint}}

{\small{\it February 2015}}

\end{flushleft}

\end{document}